\documentclass[12pt]{amsart}
\usepackage[a4paper,margin=2.5 cm]{geometry}
\address{\noindent
Pierre Youssef,\newline
Universit\'{e} Paris-Est Marne-La-Vall\'ee \newline
Laboratoire d'Analyse et Math\'{e}matiques Appliqu\'ees,\newline
5, boulevard Descartes,
Champs sur Marne, \newline
77454 Marne-la-Vall\'{e}e, Cedex 2, France \newline
\texttt{e-mail: \small pierre.youssef@univ-mlv.fr}
}

\date{}
\usepackage{amsthm,hyperref,lmodern,mathrsfs}

\usepackage[latin1]{inputenc}
\usepackage[T1]{fontenc}
\usepackage{graphicx}
\usepackage[frenchb,english]{babel}
\usepackage{calc}
\usepackage{amsmath,amsthm}
\usepackage{graphicx}
\usepackage{url}
\usepackage{times, amssymb, amscd, mathrsfs, graphicx, color}
\usepackage{enumerate}

\newtheorem{theo}{Theorem}[section]

\newtheorem{coro}[theo]{Corollary}
\newtheorem{lem}[theo]{Lemma}

\newtheorem{Rq}[theo]{Remark}
\newtheorem*{theoB}{Theorem B}
\newtheorem*{theoA}{Theorem A}
\newtheorem*{theoC}{Theorem C}
\newtheorem*{theoD}{Theorem D}
\newtheorem*{theoE}{Theorem E}

\title[]{Restricted invertibility and the Banach-Mazur \\ distance to the cube}
\author{Pierre Youssef}
\begin{document}

\maketitle

\begin{abstract}
We prove a normalized version of the restricted invertibility
principle obtained by Spielman-Srivastava in \cite{MR2956233}.
Applying this result, we get a new proof of the proportional
Dvoretzky-Rogers factorization theorem recovering the best current
estimate in the symmetric setting while we improve the best known result 
in the nonsymmetric case. As a consequence, we slightly improve the estimate
for the Banach-Mazur distance to the cube: the distance of every
$n$-dimensional normed space from $\ell_{\infty }^n$ is at most
$(2n)^{\frac{5}{6}}$. Finally, using tools from the work of
Batson-Spielman-Srivastava in \cite{batson-spielman-srivastava}, we
give a new proof for a theorem of Kashin-Tzafriri \cite{kashin-tzafriri} 
on the norm of restricted matrices.
\end{abstract}

\section{Introduction}

Given an $n\times m$ matrix $U$, viewed as an operator from
$\ell_2^m$ to $\ell_2^n$, the restricted invertibility problem asks
if we can extract a large number of linearly independent columns of
$U$ and provide an estimate for the norm of the restricted inverse.
If we write $U_{\sigma }$ for the restriction of $U$ to the columns
$Ue_i$, $i\in \sigma\subset \{ 1,\ldots ,m\}$, we want to find a
subset $\sigma $, of cardinality $k$ as large as possible, such that
$\Vert U_{\sigma } x\Vert_2 \geqslant c\Vert x\Vert_2$ for all $x\in
\mathbb{R}^{\sigma }$ and to estimate the constant $c$ (which will
depend on the operator $U$). This question was studied by
Bourgain-Tzafriri \cite{b-tz} who obtained a result for square
matrices: \vskip 0.3cm Given an $n\times n$ matrix $T$ (viewed as an
operator on $\ell_2^n$) whose columns are of norm one, there exists
$\sigma \subset \{1,\ldots ,n\}$ with $\vert \sigma\vert \geqslant d
\frac{n}{\Vert T \Vert^2}$ such that $ \Vert T_\sigma x\Vert_2
\geqslant c\Vert x\Vert_2$ for all $x\in \mathbb{R}^\sigma$, where
$d, c>0$ are absolute constants. \vskip 0.3cm

Here and in the rest of the paper, $\Vert \cdot\Vert_2$ denotes the
Euclidean norm. 
For any matrix
$A$,  $\Vert A\Vert$ denotes its operator norm seen as an operator 
on $l_2$ and 
$\Vert A\Vert_{{\rm HS}}$ denotes the Hilbert-Schmidt norm, i.e.
$$\Vert A\Vert_{{\rm HS}}= \sqrt{Tr(A\cdot A^*)} = \left(\sum_i \Vert
C_i\Vert_2^2\right)^{1/2}$$ where $C_i$ are the columns of $A$. 
Given $\sigma\subset \{1,...,m\}$, we denote $U_{\sigma}$ the restriction 
of $U$ to the columns with indices in $\sigma$ i.e $U_{\sigma}=UP_{\sigma}^t$ 
where $P_{\sigma}: \mathbb{R}^m\longrightarrow \mathbb{R}^{\sigma}$ is the canonical 
coordinate projection. 


In \cite{MR1826503}, Vershynin generalized this result for 
rectangular matrices and improved the estimate for the
size of the subset. Using a technical iteration scheme based on the
previous result of Bourgain-Tzafriri, combined with a theorem of
Kashin-Tzafriri which we will discuss in the last section, he
obtained the following : \vskip 0.3cm 

\begin{theoA}\label{vershynin}
Let $U$ be an $n\times m$ matrix and denote $\widetilde{U}$ the 
matrix $U$ with normalized columns. For any $\varepsilon \in (0,1)$, there exists 
$\sigma\subset\{1,...,m\}$ with  $$\vert \sigma\vert \geqslant \left[(1-\varepsilon)\frac{ \Vert
U\Vert_{{\rm HS}}^2}{\Vert U\Vert^2}\right] $$ such that 
$$
c_1(\varepsilon)\leqslant s_{\min}(\widetilde{U}_{\sigma})\leqslant 
s_{\max}(\widetilde{U}_{\sigma})\leqslant c_2(\varepsilon)
$$
\end{theoA}


One can easily check that, when $U$ is a square matrix, this is a generalization of the Bourgain-Tzafriri
theorem, which was previously only proved for a fixed value of
$\varepsilon$. The constants $c_1(\varepsilon)$ and $c_2(\varepsilon)$ play a crucial role in
applications and finding the right dependence is an important problem. 
Let us mention that in this paper, we will be interested only in the estimate of the smallest singular value which is 
the part related to the restricted invertibility principle.

Back to the original restricted invertibility problem, a recent work
of Spielman-Srivastava \cite{MR2956233} 
provides the best known estimate for the norm
of the inverse matrix. Their proof uses a new deterministic method
based on linear algebra, while the previous works on the subject
employed probabilistic, combinatorial and functional-analytic
arguments.

More precisely, Spielman-Srivastava proved the following:

\begin{theoB}[Spielman-Srivastava]\label{srivastava}
Let $U$ be an $n\times m$ matrix. For any $\varepsilon \in (0,1)$, there exists 
$\sigma\subset\{1,...,m\}$ with  $$\vert \sigma\vert \geqslant \left[(1-\varepsilon)^2\frac{ \Vert
U\Vert_{{\rm HS}}^2}{\Vert U\Vert^2}\right] $$ such that 
$$
s_{\min}(U_{\sigma})\geqslant \varepsilon \frac{\Vert U\Vert_{\rm}}{\sqrt{m}}
$$
\end{theoB}


In the applications, one might need to extract multiples of the
columns of the matrix. Adapting the proof of Spielman-Srivastava,
we will generalize the restricted invertibility theorem for any rectangular matrix and,
under some conditions, for any choice of multiples.

If $D$ is an $m\times m$ diagonal matrix with diagonal entries
$(\alpha_j)_{j\leqslant m}$, we set $\Gamma_D:=\{j\leqslant m\mid \
\alpha_j \neq 0\}$ and for $\sigma\subset \{1,...,m\}$ we 
write $D_{\sigma}^{-1}$ for the restricted
inverse of $D$ i.e the diagonal matrix whose diagonal entries are
the inverses of the respective entries of $D$ for indices in
$\sigma$ and zero elsewhere. The main result of this paper is the
following:

\begin{theo}\label{theoreme-principal}
Given an $n\times m$ matrix $U$ and a diagonal $m\times m$ matrix
$D$ with $(\alpha_j)_{j\leqslant m}$ on its diagonal, with the
property that $ {\rm Ker}(D)\subset {\rm Ker}(U)$, then for any
$\varepsilon\in (0,1)$ there exists $\sigma \subset \Gamma_D$
with $$\vert \sigma\vert\geqslant \left[(1-\varepsilon)^2\frac{ \Vert
U\Vert_{{\rm HS}}^2}{\Vert U\Vert^2}\right] $$ such that
\[s_\mathrm{min}\left(U_{\sigma}D_{\sigma}^{-1}\right)>
\frac{\varepsilon \|U\|_{{\rm HS}}}{\|D\|_{{\rm HS}}},\] where
$s_\mathrm{min}$ denotes the smallest singular value.
\end{theo}

Note that given a matrix $U$, if we take $D$ to be the identity operator, we recover
Theorem~B. Taking 
$D$ the diagonal matrix with diagonal
entries the norms of the columns of $U$, it is easy to see that we recover 
the "normalized" restricted invertibility part of Theorem~A 
with $c_1(\varepsilon)=\varepsilon$.

In Section 2, we give the proof of the main result. In section 3, 
we use Theorem~\ref{theoreme-principal} to give 
an alternative proof for the
proportional Dvoretzky-Rogers factorization; in the symmetric case, 
we recover the best known dependence and improve the constants involved 
which allows us to improve the estimate of the Banach-Mazur distance to the 
cube; while in the nonsymmetric case, we improve the best known dependence 
for the proportional Dvoretzky-Rogers factorization. 
Finally, in Section 4 we give a new proof of a theorem due to
Kashin-Tzafriri \cite{kashin-tzafriri} which deals with the norm of coordinate projections
of a matrix; our proof slightly improves the result of Kashin-Tzafriri and 
has the advantage of producing a deterministic algorithm.

\section{Proof of Theorem\,\ref{theoreme-principal}}

The proof of Theorem\,\ref{theoreme-principal} will be an adaptation of the argument 
used by Spielman-Srivastava \cite{MR2956233} in order to prove Theorem~B.

Since the rank and the eigenvalues of $(U_\sigma
D_\sigma^{-1})^t\cdot (U_\sigma D_\sigma^{-1})$ and $(U_\sigma
D_\sigma^{-1})\cdot (U_\sigma D_\sigma^{-1})^t$ are the same, it
suffices to prove that $(U_\sigma D_\sigma^{-1})\cdot (U_\sigma
D_\sigma^{-1})^t$ has rank equal to $k=\vert \sigma\vert$ and its
smallest positive eigenvalue is greater than
$\varepsilon^2\frac{\Vert U\Vert_{{\rm HS}}^2}{\Vert D\Vert_{{\rm
HS}}^2}$. Note that $$(U_\sigma D_\sigma^{-1})\cdot (U_\sigma
D_\sigma^{-1})^t =\displaystyle \sum_{j\in\sigma}
\left(\frac{Ue_j}{\alpha_j}\right)\cdot \left(\frac{Ue_j}{\alpha_j}\right)^t$$ We are going to
construct the matrix $A_k=\sum_{j\in\sigma}
\left(\frac{Ue_j}{\alpha_j}\right)\cdot \left(\frac{Ue_j}{\alpha_j}\right)^t$ by iteration. We
begin by setting $A_0=0$ and at each step we will be adding a rank
one matrix $\left(\frac{Ue_j}{\alpha_j}\right)\cdot \left(\frac{Ue_j}{\alpha_j}\right)^t$ for
a suitable $j$, which will give a new positive eigenvalue. This will
guarantee that the vector $\frac{Ue_j}{\alpha_j}$ chosen in each step
is linearly independent from the previous ones.

If $A$ and $B$ are symmetric matrices, we write $A\preceq B$ if
$B-A$ is a positive semidefinite matrix. Recall the Sherman-Morrison
Formula which will be needed in the proof. For any invertible matrix
$A$ and any vector $v$ we have
$$(A+v\cdot v^{t})^{-1}=A^{-1}-\frac{A^{-1}v\cdot v^{t}A^{-1}}{1+v^{t}A^{-1}v}.$$
We will also apply the following lemma which appears as Lemma 6.3 in
\cite{srivastava}:

\begin{lem}
Suppose that $A\succeq 0$ has $q$ nonzero eigenvalues, all greater
than $b'>0$. If $v\neq 0$ and
\begin{equation}\label{condition1}
v^{t} (A-b'I)^{-1}v<-1,
\end{equation}
then $A+vv^{t}$ has $q+1$ nonzero eigenvalues, all greater than
$b'$.
\end{lem}


\medskip

For any symmetric matrix $A$ and any $b>0$, we define $$\phi(A,b)={\rm
Tr}\left(U^{t}(A-bI)^{-1}U\right)$$ as the potential corresponding to
the barrier $b$. 

At each step $l$, the matrix already constructed is
denoted by $A_l$ and the barrier by $b_l$. Suppose that $A_l$ has $l$ nonzero
eigenvalues all greater than $b_l$. As mentioned before, we will try
to construct $A_{l+1}$ by adding a rank one matrix $v\cdot v^{t}$ to
$A_l$ so that $A_{l+1}$ has $l+1$ nonzero eigenvalues all greater
than $b_{l+1}=b_l-\delta$ and $\phi(A_{l+1},b_{l+1})\leqslant
\phi(A_l,b_l)$. Note that
\begin{align*}
\phi(A_{l+1},b_{l+1})&={\rm Tr}\left(U^{t}(A_l+vv^{t}-b_{l+1}I)^{-1}U\right)\\
&={\rm Tr}\left(U^{t}(A_{l}-b_{l+1}I)^{-1}U\right)-
{\rm Tr}\left(\frac{U^{t}(A_l-b_{l+1}I)^{-1}vv^{t}(A_l-b_{l+1}I)^{-1}U}{1+v^{t}(A_l-b_{l+1}I)^{-1}v}\right)\\
&=\phi(A_{l},b_{l+1})-\frac{v^{t}(A_l-b_{l+1}I)^{-1}UU^{t}(A_l-b_{l+1}I)^{-1}v}{1+v^{t}(A_l-b_{l+1}I)^{-1}v}.
\end{align*}
So, in order to have $\phi(A_{l+1},b_{l+1})\leqslant \phi(A_l,b_l)$, we
must choose a vector $v$ verifying
\begin{equation}\label{condition2}
-\frac{v^{t}(A_l-b_{l+1}I)^{-1}UU^{t}(A_l-b_{l+1}I)^{-1}v}{1+v^{t}(A_l-b_{l+1}I)^{-1}v}\leqslant
\phi(A_{l},b_{l})-\phi(A_{l},b_{l+1}).
\end{equation}
Since $v^{t}(A_l-b_{l+1}I)^{-1}UU^{t}(A_l-b_{l+1}I)^{-1}v$ and
$\phi(A_{l},b_{l})-\phi(A_{l},b_{l+1})$ are positive, choosing $v$
verifying condition (\ref{condition1}) with $b'=b_{l+1}$ and condition (\ref{condition2}) is
equivalent to choosing $v$ which satisfies the following:
$$
v^{t}(A_l-b_{l+1}I)^{-1}UU^{t}(A_l-b_{l+1}I)^{-1}v\leqslant
\left(\phi(A_{l},b_{l})-\phi(A_{l},b_{l+1})\right)\left(-1-v^{t}(A_l-b_{l+1}I)^{-1}v\right)
$$
Since $UU^{t}\preceq \Vert U\Vert^{2}Id$ and
$(A_l-b_{l+1}I)^{-1}$ is symmetric, it is sufficient to choose $v$
so that
\begin{equation}\label{condition4}
v^{t}(A_l-b_{l+1}I)^{-2}v \leqslant \frac{1}{ \Vert U\Vert^{2}}
\left(\phi(A_{l},b_{l})-\phi(A_{l},b_{l+1})\right)\left(-1-v^{t}(A_l-b_{l+1}I)^{-1}v\right)
\end{equation}

Recall the notation $\Gamma_D:=\{j\leqslant m\mid \alpha_j\neq 0\}$ where
$(\alpha_j)_{j\leqslant m}$ are the diagonal entries of $D$. Since
we have assumed that $ {\rm Ker}(D)\subset {\rm Ker}(U)$, we have
$$\Vert U\Vert_{{\rm {\rm HS}}}^{2}= \displaystyle \sum_{j\leqslant
m} \left\Vert Ue_j\right\Vert_2^{2} = \sum_{j\in \Gamma_D} \left\Vert
Ue_j\right\Vert_2^{2} \leqslant \vert \Gamma_D\vert \cdot \Vert
U\Vert^{2},$$ and thus $\vert \Gamma_D\vert \geqslant \frac{ \Vert
U\Vert_{{\rm HS} }^{2} }{ \Vert U \Vert^{2}}$. At each step, we
will select a vector $v$ satisfying (\ref{condition4}) among
$(\frac{Ue_j}{ \alpha_j})_{j\in \Gamma_D}$. Our task therefore is to find
$j\in \Gamma_D$ such that
\begin{equation}\label{condition5}
(Ue_j)^{t}(A_l-b_{l+1}I)^{-2}Ue_j \leqslant
\frac{\phi(A_{l},b_{l})-\phi(A_{l},b_{l+1})}{ \Vert U\Vert^{2}}
\left(-\alpha_j^{2}-(Ue_j)^{t}(A_l-b_{l+1}I)^{-1}Ue_j\right)
\end{equation}
The existence of such a $j\in \Gamma_D$ is guaranteed by the fact that
condition (\ref{condition5}) holds true if we take the sum over all
$(\frac{Ue_j}{\alpha_j})_{j\in D}$. The hypothesis $ {\rm
Ker}(D)\subset  {\rm Ker}(U)$ implies that:
\begin{itemize}
\item[$\bullet$] $\displaystyle \sum_{j\in \Gamma_D}(Ue_j)^{t}(A_l-b_{l+1}I)^{-2}Ue_j = {\rm
Tr}\left(U^{t}(A_l-b_{l+1}I)^{-2}U\right)$,
\item[$\bullet$] $\displaystyle \sum_{j\in \Gamma_D} (Ue_j)^{t}(A_l-b_{l+1}I)^{-1}Ue_j={\rm
Tr}\left(U^{t}(A_l-b_{l+1}I)^{-1}U\right)$.
\end{itemize}
Therefore it is enough to prove that, at each step, one has
\begin{equation}\label{condition6}
{\rm Tr}(U^{t}(A_l-b_{l+1}I)^{-2}U) \leqslant
\frac{\phi(A_{l},b_{l})-\phi(A_{l},b_{l+1})}{ \Vert U\Vert^{2}} \left(
-\Vert D\Vert_{{\rm HS}}^{2} - \phi(A_{l},b_{l+1})\right)
\end{equation}

The rest of the proof is similar to the one in \cite{srivastava}.
One just needs to replace $m$ by $\Vert D\Vert_{{\rm HS}}^{2}$. For
the sake of completeness, we include the proof. The next lemma  will
determine the conditions required at each step in order to prove
(\ref{condition6}).

\begin{lem}\label{lemma2-invertibility}
Let $A$ be an $n\times n$ symmetric positive semidefinite matrix. 
Suppose that $A$ has $l$ nonzero eigenvalues all greater than
$b_l$, and write $Z$ for the orthogonal projection onto the kernel
of $A$. If
\begin{equation}\label{condition1-lemma}
\phi(A,b_l)\leqslant - \Vert D\Vert_{{\rm HS}}^{2} - \frac{ \Vert
U\Vert^{2}}{\delta}
\end{equation}
and
\begin{equation}\label{condition2-lemma}
0<\delta < b_l\leqslant \delta \frac{ \Vert ZU\Vert_{{\rm
HS}}^{2}}{\Vert U\Vert^{2}},
\end{equation}
then there exists $i\in \Gamma_D$ such that $A' :=A + \left(\frac{ Ue_i}{
\alpha_i}\right)\cdot \left(\frac{ Ue_i}{ \alpha_i}\right)^t$ has $l+1$ nonzero
eigenvalues all greater than $b_{l+1}:= b_l - \delta$ and
$\phi(A',b_{l+1})\leqslant \phi(A,b_l)$.
\end{lem}

\begin{proof}[Proof]
As mentioned before, it is enough to prove inequality
(\ref{condition6}). We set
$\Delta_l:=\phi(A,b_l)-\phi(A',b_{l+1})$. By
(\ref{condition1-lemma}), we get $$\phi(A,b_{l+1})\leqslant -\Vert
D\Vert_{{\rm HS}}^{2} - \frac{ \Vert U\Vert^{2}}{\delta} -
\Delta_l.$$ Inserting this in (\ref{condition6}), we see that it is
sufficient to prove the following inequality:
\begin{equation}\label{condition7}
{\rm Tr}\left(U^{t}(A-b_{l+1}I)^{-2}U\right) \leqslant
\Delta_l\left( \frac{\Delta_l}{ \Vert U\Vert^{2}} +
\frac{1}{\delta}\right).
\end{equation}
Now, denote by $P$ the orthogonal projection onto the image of
$A$. We set $$\phi^{P}(A,b_l):={\rm
Tr}\left(U^{t}P(A-b_lI)^{-1}PU\right)\quad\hbox{and}\quad
\Delta_l^{P}:=\phi^{P}(A,b_l)-\phi^{P}(A,b_{l+1})$$ and use similar
notation for $Z$. Since $P$, $Z$ and $A$ commute, one can
write $$\Delta_l=\Delta_l^{P}+\Delta_l^{Z}\quad\hbox{and}\quad
\phi(A,b_l)=\phi^{P}(A,b_l)+\phi^{Z}(A,b_l).$$ Note that:
\begin{align*}
(A - b_l I)^{-1} - (A - b_{l+1} I)^{-1} &=
(A - b_l I)^{-1} (b_l I - A + A - b_{l+1} I) (A - b_{l+1} I)^{-1}\\
& =\delta (A - b_l I)^{-1}  (A - b_{l+1} I)^{-1}
\end{align*}
and since $P(A-b_lI)^{-1}P$ and $P(A-b_{l+1}I)^{-1}P$ are positive
semidefinite, we have:
$$
U^{t}P(A-b_lI)^{-1}PU-U^tP(A-b_{l+1}I)^{-1}PU \succeq \delta U^{t}P(A-b_{l+1}I)^{-2}PU.
$$
Inserting this in (\ref{condition7}), it is enough to prove that:
\begin{equation*}\label{condition8}
{\rm Tr}\left(U^{t}Z(A-b_{l+1}I)^{-2}Z U\right)\leqslant
\Delta_l \left( \frac{\Delta_l}{ \Vert U\Vert^{2}} +
\frac{1}{\delta}\right)-\frac{\Delta_l^{P}}{\delta}.
\end{equation*}
Since $AZ=0$, we have:
\begin{itemize}
\item[$\bullet$] ${\rm Tr}(U^{t}Z(A-b_{l+1}I)^{-2}Z U)= \frac{ \Vert ZU\Vert_{{\rm
HS}}^{2}}{b_{l+1}^{2}}$ and
\item[$\bullet$] $\Delta_l^{Z}= -\frac{ \Vert ZU\Vert_{{\rm HS}}^{2}}{b_{l}}
+\frac{ \Vert ZU\Vert_{{\rm HS}}^{2}}{b_{l+1}}=\delta \frac{ \Vert
ZU\Vert_{{\rm HS}}^{2}}{b_lb_{l+1}}$,
\end{itemize}
so taking into account the fact that $\Delta_l \geqslant
\Delta_l^{Z}\geqslant 0$, it remains to prove the following:
\begin{equation}\label{condition9}
\frac{ \Vert ZU\Vert_{{\rm HS}}^{2}}{b_{l+1}^{2}} \leqslant
\delta^{2} \frac{ \Vert ZU\Vert_{{\rm HS}}^{4}}{\Vert U\Vert_{2
}^{2} b_l^{2}b_{l+1}^{2}} +\frac{ \Vert ZU\Vert_{{\rm
HS}}^{2}}{b_lb_{l+1}}.
\end{equation}
By Hypothesis (\ref{condition2-lemma}), this last inequality follows
by
\begin{equation}
\frac{ \Vert ZU\Vert_{{\rm HS}}^{2}}{b_{l+1}^{2}} \leqslant \delta
\frac{ \Vert ZU\Vert_{{\rm HS}}^{2}}{b_lb_{l+1}^{2}} +\frac{ \Vert
ZU\Vert_{{\rm HS}}^{2}}{b_lb_{l+1}},
\end{equation}
which is trivially true since $b_{l+1}=b_l-\delta$.
\end{proof}

We are now able to complete the proof of
Theorem~\ref{theoreme-principal}. To this end, we must verify that
conditions (\ref{condition1-lemma}) and (\ref{condition2-lemma})
hold at each step. At the beginning we have $A_0=0$ and $Z=Id$, so
we must choose a barrier $b_0$ such that:
\begin{equation}\label{conditionfinale1}
-\frac{\Vert U\Vert_{{\rm HS}}^{2}}{b_0} \leqslant -\Vert
D\Vert_{{\rm HS}}^{2}-\frac{ \Vert U\Vert^{2}}{\delta}
\end{equation}
and
\begin{equation}\label{conditionfinale2}
b_0 \leqslant \delta\frac{\Vert U\Vert_{{\rm HS}}^{2}}{ \Vert
U\Vert^{2}}.
\end{equation}
We choose $$b_0 :=\varepsilon \frac{ \Vert U\Vert_{{\rm HS}}^{2}}{
\Vert D\Vert_{{\rm HS}}^{2}}\quad\hbox{and}\quad \delta:=
\frac{\varepsilon}{1-\varepsilon}\frac{ \Vert U\Vert^{2}}{\Vert
D\Vert_{{\rm HS}}^{2}},$$ and we note that (\ref{conditionfinale1})
and (\ref{conditionfinale2}) are verified. Also, at each step
(\ref{condition1-lemma}) holds because $\phi(A_{l+1},b_{l+1})\leqslant
\phi(A_l,b_l)$. Since $\Vert Z U\Vert_{{\rm HS}}^{2} $ decreases at
each step by at most $\Vert U\Vert^{2}$, the right-hand side of
(\ref{condition2-lemma}) decreases by at most $\delta$, and
therefore (\ref{condition2-lemma}) holds once we replace $b_l$ by
$b_l-\delta$.

Finally note that, after $k=(1-\varepsilon)^{2}\frac{ \Vert
U\Vert_{{\rm HS}}^{2}}{ \Vert U\Vert^{2}}$ steps, the barrier
will be
$$b_k=b_0-k\delta = \varepsilon^{2}\frac{ \Vert U\Vert_{{\rm
HS}}^{2}}{ \Vert D\Vert_{{\rm HS}}^{2}}.$$ This completes the proof.

\section{Proportional Dvoretzky-Rogers factorization}

By the classical Dvoretzky-Rogers lemma \cite{MR0033975}, 
if $X$ is an $n$-dimensional Banach space then there 
exist $x_1,...,x_m\in X$ with $m=\sqrt{n}$ such that for all scalars $(a_j)_{j\leqslant m}$
$$
\max_{j\leqslant m} \vert a_j\vert \leqslant 
\left\Vert \sum_{j\leqslant m} a_jx_j\right\Vert_X
\leqslant c\left(\sum_{j\leqslant m} a_j^2\right)^{\frac{1}{2}},
$$
where $c$ is a universal constant. Bourgain-Szarek \cite{MR947820} proved that 
the previous statement holds for $m$ proportional to $n$, and called the result 
"the proportional Dvoretzky-Rogers factorization":

\begin{theoC}[Proportional Dvoretzky-Rogers factorization]\label{th-proportional-D-R}
Let $X$ be an $n$-dimensional Banach space. $\forall \varepsilon \in (0,1)$, there exist 
$x_1,...,x_k\in X$ with $k\geqslant \left[(1-\varepsilon)n\right]$ such that for all scalars $(a_j)_{j\leqslant k}$
$$
\max_{j\leqslant k} \vert a_j\vert \leqslant 
\left\Vert \sum_{j\leqslant k} a_jx_j\right\Vert_X
\leqslant c(\varepsilon)\left(\sum_{j\leqslant k} a_j^2\right)^{\frac{1}{2}},
$$
where $c(\varepsilon)$ is a constant depending on $\varepsilon$. Equivalently, the identity operator 
$i_{2,\infty}: l_2^k\longrightarrow l_\infty^k$ can be written $i_{2,\infty}= \alpha \circ \beta$ with 
$\beta : l_2^k \longrightarrow X, \alpha : X\longrightarrow l_\infty^k$ and $\Vert \alpha \Vert \cdot \Vert \beta\Vert \leqslant c(\varepsilon)$.
\end{theoC}

Finding the right dependence on $\varepsilon$ is an important problem and the optimal 
result is not known yet. In \cite{MR1081810}, Szarek showed that the dependence cannot be 
better than $c\varepsilon^{-\frac{1}{10}}$. Szarek-Talagrand \cite{MR1008718} proved that the previous 
result holds with $c(\varepsilon)=c\varepsilon^{-2}$ and in \cite{MR1353450} and \cite{MR1301496} 
Giannopoulos improved the dependence to get $c\varepsilon^{-\frac{3}{2}}$ and $c\varepsilon^{-1}$. 
In all these results, a factorization for the identity operator $i_{1,2}: l_1^k\longrightarrow l_2^k$ was proven 
and by duality the factorization for $i_{2,\infty}$ was deduced. The previous proofs used some 
geometric results, technical combinatorics and Grothendieck's factorization theorem. 
Here we present a direct proof using Theorem~\ref{theoreme-principal} 
which allows us to recover the best known dependence on $\varepsilon$ 
and improve the universal constant involved. \\

Note that Theorem~C can be formulated with symmetric convex bodies. 
 In \cite{MR1796719}, Litvak and Tomczak-Jaegermann proved a nonsymmetric version 
 of the proportional Dvoretzky-Rogers factorization:

\begin{theoD}[Litvak-Tomczak-Jaegermann]
Let $K\subset \mathbb{R}^n$ be a convex body, such that $B_2^n$ is 
the ellipsoid of minimal volume containing $K$. Let $\varepsilon \in (0,1)$ 
and set $k= \left[(1-\varepsilon) n \right]$. There exist vectors $y_1,y_2,...,y_k$ 
in $K$, and an orthogonal projection $P$ in $\mathbb{R}^n$ with rank $P \geqslant k$ 
such that for all scalars $t_1,...,t_k$
$$c\varepsilon^3 \left( \sum_{j=1}^k \vert t_j\vert^2\right)^{\frac{1}{2}} 
\leqslant \left\Vert \sum_{j=1}^k t_j Py_j \right\Vert_{PK} 
\leqslant \frac{6}{\varepsilon} \sum_{j=1}^k \vert t_j\vert ,$$
where $c>0$ is a universal constant.
\end{theoD}

Using again Theorem~\ref{theoreme-principal} combined with some tools developed 
in \cite{MR947820} and \cite{MR1796719}, we will be able to improve the dependence on $\varepsilon$ 
in the previous statement.

\subsection{The symmetric case}

Let us start with the original proportional Dvoretzky-Rogers factorization. We will 
prove the following:

\begin{theo}\label{dvoretzky-rogers}
Let $X$ be an $n$-dimensional Banach space. $\forall \varepsilon \in (0,1)$, there exist 
$x_1,...,x_k\in X$ with $k\geqslant \left[(1-\varepsilon)^2n\right]$ such that for all scalars $(a_j)_{j\leqslant m}$
$$
\varepsilon \left(\sum_{j\leqslant k} a_j^2\right)^{\frac{1}{2}}\leqslant 
\left\Vert \sum_{j\leqslant k} a_jx_j\right\Vert_X
\leqslant  \sum_{j\leqslant k} \vert a_j\vert 
$$
Equivalently, the identity operator
$i_{1,2}: l_1^{k} \longrightarrow l_{2}^k$ can be written
as $i_{1,2}= \alpha \circ \beta$, where $\beta : l_{1}^k
\longrightarrow X$, $\alpha : X  \longrightarrow l_{2}^k $ and
$\Vert \alpha\Vert \cdot \Vert \beta \Vert \leqslant
\varepsilon^{-1}$.
\end{theo}

\begin{proof}

Without loss of generality, we may assume that $X= (\mathbb{R}^n, \Vert \cdot\Vert_X)$ 
and $B_2^n$ is the ellipsoid of minimal volume 
containing $B_X$. By John's theorem \cite{MR0030135} there 
exist $x_1,...,x_m$ contact points of $B_X$ with $B_2^n$ ($\Vert x_j\Vert_X
= \Vert x_j\Vert_{X^*}=\Vert x_j\Vert_2=1$) and positive scalars $c_1,...,c_m$ such that 
$$
Id=\displaystyle \sum_{j\leqslant m}c_j x_jx_j^t 
$$

Let $U=\left( \sqrt{c_1} x_1,...,\sqrt{c_m}x_m\right)$ be the $n\times m$ rectangular 
matrix whose columns are $\sqrt{c_j}x_j$ and denote $D=diag(\sqrt{c_1},...,\sqrt{c_m})$ 
the $m\times m$ diagonal matrix with $\sqrt{c_j}$ on its diagonal. It would be helpful to observe that 
$
UU^t=Id
$, thus $\Vert U\Vert =1$ and $\Vert U\Vert_{\rm HS}^2=n$.

Let $\varepsilon <1$, applying Theorem~\ref{theoreme-principal} to $U$ and $D$, 
we find $\sigma \subset \{1,...,m\}$ such that
$$
k=\vert \sigma\vert \geqslant \left[(1-\varepsilon)^2n\right]
$$
and for all $a=(a_j)_{j\leqslant m}$
\begin{equation}\label{inverse-norm-of-T}
\left\Vert U_{\sigma}D_{\sigma}^{-1} a\right\Vert_2= 
\left\Vert \sum_{j\in\sigma} a_j x_j \right\Vert_2 \geqslant \varepsilon\left(\sum_{j\in\sigma} 
\vert a_j\vert^2\right)^{\frac{1}{2}}
\end{equation}

Since $\Vert\cdot\Vert_2\leqslant \Vert\cdot\Vert_X$ and using the triangle inequality, we have 
$$
\varepsilon\left(\sum_{j\in\sigma} 
\vert a_j\vert^2\right)^{\frac{1}{2}}\leqslant 
\left\Vert \sum_{j\in\sigma} a_j x_j \right\Vert_2 \leqslant 
\left\Vert \sum_{j\in\sigma} a_j x_j \right\Vert_X \leqslant 
\sum_{j\in \sigma} \vert a_j\vert
$$

\end{proof}

Let $\mathbb{BM}_n$ denote the space of all $n$-dimensional normed
spaces $X$, known as the Banach-Mazur compactum. If $X,Y$ are in
$\mathbb{BM}_n$, the Banach-Mazur distance between $X$ and
$Y$ is defined as follows:
$$
d(X,Y)=  {\inf} \{ \Vert T\Vert \cdot \Vert T^{-1}\Vert \mid \ T \text{ is an isomorphism between $X$ and $Y$ }\}
$$

\begin{Rq}
For $K, L$ two symmetric convex bodies in $\mathbb{R}^n$, the Banach-Mazur distance 
between $K$ and $L$ is given by 
$$
d(K,L)=  {\inf}\left\{ \alpha/\beta\mid \ \beta L\subset T(K) \subset \alpha L\right\}
$$
One can easily check that this distance is coherent with the previous one as $d(X, Y) = d(B_X , B_Y)$.
\end{Rq}

As a direct application of the previous result, we have
\begin{coro}
Let $X$ be an $n$-dimensional Banach space. For any $\varepsilon \in (0,1)$,
there exists $Y$ a subspace of $X$ of dimension $k\geqslant \left[(1-\varepsilon)^2n\right]$ such 
that $d(Y,l_1^k)\leqslant \frac{\sqrt{n}}{\varepsilon}$.
\end{coro}

\subsection{The nonsymmetric case} 

Let us now turn to the nonsymmetric version of Theorem~\ref{dvoretzky-rogers}. 
We will prove the following:

\begin{theo}
Let $K\subset \mathbb{R}^n$ be a convex body, such that $B_2^n$ is 
the ellipsoid of minimal volume containing $K$.
$\forall \varepsilon \in (0,1)$, there exist $x_1,...,x_k$ with $k\geqslant \left[(1-\varepsilon)n\right]$ 
contact points and there exists $P$ an orthogonal projection of rank $\geqslant k$ 
such that for all $(a_j)_{j\leqslant k}$
$$\frac{\varepsilon^2}{16} \left( \sum_{j=1}^k \vert a_j\vert^2\right)^{\frac{1}{2}}
\leqslant \left\Vert \sum_{j=1}^k a_j Px_j\right\Vert_{PK}
\leqslant \frac{4}{\varepsilon} \sum_{j=1}^k \vert a_j\vert$$ 
\end{theo}

\begin{proof}
By John's Theorem \cite{MR0030135}, we get an identity decomposition in $\mathbb{R}^n$
$$
Id= \sum_{j=1}^m  c_j x_jx_j^t 
$$
where $x_1,...,x_m$ are contact points of $K$ and $B_2^n$ and $(c_j)_{j\leqslant m}$ 
positive scalars. 

Similarly to the proof of Theorem~\ref{dvoretzky-rogers}, we find $\sigma_1\subset \{1,...,m\}$ such that
$$s=\vert \sigma_1\vert \geqslant \left(1-\frac{\varepsilon}{4}\right)^2n \geqslant (1-\frac{\varepsilon}{2})n$$
and for all $a=(a_j)_{j\leqslant m}$
\begin{equation}\label{inverse norm of T}
\left\Vert U_{\sigma_1}D_{\sigma_1}^{-1} a\right\Vert_2= 
\left\Vert \sum_{j\in\sigma_1} a_j x_j \right\Vert_2 \geqslant \frac{\varepsilon}{4} \left(\sum_{j\in\sigma_1} 
\vert a_j\vert^2\right)^{\frac{1}{2}}
\end{equation}

Define $Y={\rm span}\{x_j\}_{j\in \sigma_1}$. We will now use the argument of Litvak and Tomczak-Jaegermann \cite{MR1796719} to construct 
the projection $P$. 
First partition $\sigma_1$ into $\left[\frac{\varepsilon}{2}s\right]$ disjoint subsets $A_l$ of equal size.
Clearly $$\left\vert A_l\right\vert \leqslant \left[\frac{s}{[\frac{\varepsilon}{2}s]}\right] +1 
\leqslant \left[\frac{2}{\varepsilon}\cdot\frac{\frac{\varepsilon}{2}s}{[\frac{\varepsilon}{2}s]}\right] +1
\leqslant\left[\frac{4}{\varepsilon}\right] +1$$

Let $z_l= \sum_{i\in A_l}x_i$ and take $P :Y\longrightarrow Y$ the orthogonal projection onto ${\rm span}\{z_l\}^{\bot}$. 
For every $l$, we have $Pz_l=0$ so that for $j\in A_l$ we can write

$$-Px_j= \sum_{i\in A_l, i\neq j} Px_i = \left(\left\vert A_l\right\vert -1\right) \cdot \frac{1}{\left\vert A_l\right\vert-1} \sum_{i\in A_l, i\neq j} Px_i$$
We deduce that for every $l$ and every $j\in A_l$, we have 
\begin{equation}\label{eq-nonsym-DR}
-Px_j \in \left(\left\vert A_l\right\vert -1\right) PK\subset \frac{4}{\varepsilon} PK
\end{equation}

Let $T:\mathbb{R}^{\vert \sigma_1\vert}\longrightarrow Y$ a linear operator defined by $Te_j= x_j$ for all $j\in\sigma_1$, 
where $(e_j)_{j\in\sigma_1}$ denotes the canonical basis of $\mathbb{R}^{\vert \sigma_1\vert}$ and $Y$ is 
equipped with the euclidean norm. 
Since $(x_j)_{j\leqslant s}$ are 
linearly independent, $T$ is an isomorphism. Moreover, by (\ref{inverse norm of T}), we have 
$\Vert T^{-1}\Vert\leqslant \frac{4}{\varepsilon}$. 
Take $P'= T^{-1}PT$ and $P''$ the orthogonal projection onto $\left({\rm Ker } P'\right)^{\perp}$. It is easy 
to check that $P''P'=P''$ and 
$$k={\rm rank }P'' ={\rm rank }P \geqslant \left(1-\frac{\varepsilon}{2}\right) s \geqslant \left(1-\varepsilon\right)n$$
 For all scalars $(a_j)_{j\in \sigma_1}$, 
 
 \begin{align*}
\left\Vert \sum_{j\in\sigma_1} a_j Px_j\right\Vert_2 
&= \left\Vert \sum_{j\in\sigma_1} a_jPTe_j\right\Vert_2\\
&=\left\Vert \sum_{j\in\sigma_1} T\left(a_jP'e_j\right)\right\Vert_2\\
&\geqslant \frac{1}{\Vert T^{-1}\Vert}\cdot \left\Vert \sum_{j\in\sigma_1} a_j P'e_j\right\Vert_2\\
&\geqslant \frac{\varepsilon}{4} \cdot \left\Vert \sum_{j\in\sigma_1} a_j P''e_j\right\Vert_2
 \end{align*}
 Now take $U=\left( P''e_1,...,P''e_s\right)$ the $s\times s$ matrix whose columns are $(P''e_j)$. 
 Apply Theorem~\ref{theoreme-principal} with $U$ and $Id$ as diagonal matrix and $\frac{\varepsilon}{4}$ 
 as parameter, then there exists $\sigma \subset \sigma_1$ of size

$$\vert \sigma\vert \geqslant \left(1-\frac{\varepsilon}{4}\right)^2 s \geqslant (1-\varepsilon)n$$
such that for all scalars $(a_j)_{j\in \sigma}$, 

$$
\left\Vert \sum_{j\in\sigma} a_jP''e_j\right\Vert_2 \geqslant \frac{\varepsilon}{4} 
\left(\sum_{j\in\sigma}\vert a_j\vert^2\right)^{\frac{1}{2}}
$$
This gives us the following

$$
\left\Vert \sum_{j\in\sigma} a_jPx_j\right\Vert_2\geqslant 
\frac{\varepsilon}{4}\cdot \left\Vert \sum_{j\in\sigma} a_jP''e_j\right\Vert_2 
\geqslant \frac{\varepsilon^2}{16} \left(\sum_{j\in\sigma}\vert a_j\vert^2\right)^{\frac{1}{2}}
$$
On the other hand, since $K\subset B_2^n$ we have $PK\subset B_2^k$ and therefore 

$$
\left\Vert \sum_{j\in\sigma} a_jPx_j\right\Vert_2 \leqslant \left\Vert \sum_{j\in\sigma} a_jPx_j\right\Vert_{PK}
$$

Denoting $A=-PK\cap PK$ which is a centrally symmetric convex body and using 
(\ref{eq-nonsym-DR}) alongside the triangle inequality, one can write 

$$
\left\Vert \sum_{j\in\sigma} a_jPx_j\right\Vert_{A}\leqslant \frac{4}{\varepsilon} \sum_{j\in\sigma}\vert a_j\vert
$$

Finally, we have

$$\frac{\varepsilon^2}{16} \left(\sum_{j\in\sigma}\vert a_j\vert^2\right)^{\frac{1}{2}}
\leqslant \left\Vert \sum_{j\in\sigma} a_jPx_j\right\Vert_2
\leqslant \left\Vert \sum_{j\in\sigma} a_jPx_j\right\Vert_{PK}
\leqslant \left\Vert \sum_{j\in\sigma} a_jPx_j\right\Vert_{A}\leqslant \frac{4}{\varepsilon} \sum_{j\in\sigma}\vert a_j\vert$$

\end{proof}

One can interpret the previous result geometrically as follows:
\begin{coro}
Let $K\subset \mathbb{R}^n$ be a convex body such that $B_2^n$ is 
the ellipsoid of minimal volume containing $K$.
For any $\varepsilon \in (0,1)$, there exists $P$ an orthogonal projection of rank 
$k\geqslant \left[(1-\varepsilon)n\right]$ such that 
$$
\frac{\varepsilon}{4} B_1^k\subset PK\subset \frac{16}{\varepsilon^2} B_2^k.
$$
Moreover, $d(PK,B_1^k)\leqslant \frac{64\sqrt{n}}{\varepsilon^3}$.\\

By duality, this means that there exists a subspace $E\subset \mathbb{R}^n$ 
of dimension $k\geqslant \left[(1-\varepsilon)n\right]$ such that
$$
\frac{\varepsilon^2}{16} B_2^k \subset K\cap E\subset \frac{4}{\varepsilon} B_\infty^k.
$$
Moreover, $d(K\cap E,B_\infty^k)\leqslant \frac{64\sqrt{n}}{\varepsilon^3}$.
\end{coro}

\subsection{Estimate of the Banach-Mazur distance to the Cube}

In \cite{MR947820}, Bourgain-Szarek showed how 
to estimate the Banach-Mazur distance to the cube once 
a proportional Dvoretzky-Rogers factorization is proven. This technique 
was again used in \cite{MR1353450} and \cite{MR1008718}. Since we 
are able to obtain a proportional Dvoretzky-Rogers factorization 
with a better constant, using the same argument we will recover 
the best known asymptotic for the Banach-Mazur distance to the cube 
and improve the constants involved. Let us start defining 
$$
R_\infty^n=\max\left\{ d(X,l_\infty^n)\ \vert\  X\in\mathbb{BM}_n\right\}
$$
Similarly one can define $R_1^n$, and since the Banach-Mazur distance 
is invariant by duality then $R_1^n=R_\infty^n$. It follows from John's theorem 
\cite{MR0030135} that the diameter of $\mathbb{BM}_n$ is less than 
$n$ and therefore a trivial estimate 
is $R_\infty^n\leqslant n$. In \cite{MR1081810}, Szarek showed 
the existence of an $n$-dimensional Banach space $X$ such that 
$d(X,l_\infty^n)\geqslant c\sqrt{n} \log(n)$. Bourgain-Szarek proved 
in \cite{MR947820} that $R_\infty^n\leqslant o(n)$ while Szarek-Talagrand \cite{MR1008718} 
and Giannopoulos \cite{MR1353450} improved this upper bound to $cn^{\frac{7}{8}}$ and 
$cn^\frac{5}{6}$ respectively. Here, we will prove the following estimate:

\begin{theo}\label{distance-au-cube}
Let $X$ be an $n$-dimensional Banach space. 
Then $$
d(X,l_1^n)\leqslant 2^{\frac{4}{3}}\sqrt{n} \cdot d(X,l_2^n)^{\frac{2}{3}}.
$$
\end{theo}

\begin{proof}[proof]
We denote $d_X=d(X,l_2^n)$. In order to bound $d(X,l_1^n)$, we need to define an isomorphism 
$T: l_1^n\longrightarrow X$ and estimate $\Vert T\Vert\cdot\Vert T^{-1}\Vert$. 
A natural way is to find a basis of $X$ and then define $T$ the operator which 
sends the canonical basis of $\mathbb{R}^n$ to this basis of $X$. The main 
idea is to find a "large" subspace $Y$ of $X$ which is "not too far" from $l_1$ 
(actually more is needed), then complement the basis of $Y$ to obtain a basis 
of $X$. Finding the "large" subspace is the heart of the method and is 
basically given by the proportional Dvoretzky-Rogers factorization. 
The proof is mainly divided in four steps:
\vskip 0.3cm
\textbf{-First step}: Place $B_X$ into a "good" position and choose the right euclidean structure.\\
Since the Banach-Mazur distance is invariant under linear transformation, 
we may change the position of $B_X$. Therefore without loss of 
generality we may assume that $X=(\mathbb{R}^n, \Vert \cdot\Vert_X)$ and 
$B_2^n$ is the ellipsoid of minimal volume containing $B_X$. Denote 
also $\mathcal{E}$ the distance ellipsoid i.e 
\begin{equation}\label{eq-position1}
\frac{1}{d_X}\mathcal{E}\subset B_X\subset \mathcal{E}
\end{equation}
The ellipsoid $\mathcal{E}$ can be defined as 
$$
\mathcal{E}=\left\{ x\in\mathbb{R}^n/\ \sum_{j=1}^n \alpha_i^2 \langle x,v_j\rangle^2 \leqslant 1\right\},
$$
where $v_j$ is an orthonormal basis (in the standard sense) of $\mathbb{R}^n$ and $\alpha_j$ positive scalars. 
To take into consideration the two euclidean structures, we will define the following ellipsoid 
$$
\mathcal{E}_1=\left\{ x\in\mathbb{R}^n/\ \sum_{j=1}^n \frac{1}{2}\left(1+\alpha_i^2\right) \langle x,v_j\rangle^2 \leqslant 1\right\}.
$$
It is easy to check that 
\begin{equation}\label{eq-position2}
B_2^n\cap\mathcal{E}\subset \mathcal{E}_1\subset \sqrt{2}B_2^n\cap\mathcal{E}
\end{equation}

Therefore
\begin{equation}\label{eq-position3}
\frac{1}{\sqrt{2} d_X} \mathcal{E}_1\subset B_X\subset \mathcal{E}_1
\end{equation}
 
\textbf{-Second step}: Let $\varepsilon >0$ and set $k= (1-2\varepsilon)n$. 
Similarly to the proof of Theorem~\ref{dvoretzky-rogers}, we find 
$x_1,...,x_k$ in $X$ such that for all scalars $(a_j)_{j\leqslant k}$
\begin{equation}\label{eq-DR}
\varepsilon \left(\sum_{j\leqslant k} a_j^2\right)^{\frac{1}{2}}\leqslant 
\left\Vert \sum_{j\leqslant k} a_jx_j\right\Vert_{2}\leqslant 
\left\Vert \sum_{j\leqslant k} a_jx_j\right\Vert_{X}
\leqslant  \sum_{j\leqslant k} \vert a_j\vert \\ 
\end{equation}

Note that $(x_j)_{j\leqslant k}$ are linearly independent and are a good candidate 
to be part of the basis of $X$. \vskip 0.3cm 
\textbf{-Third step}: To form a basis of $X$, we simply take $y_{k+1},..,y_n$ 
an orthogonal basis in the $\mathcal{E}_1$-sense of span$\left\{(x_j)_{j\leqslant k}\right\}^{\bot}$ (where the $\bot$ is in the $\mathcal{E}_1$-sense) such that 
$\Vert y_j\Vert_{\mathcal{E}_1} = \frac{1}{\sqrt{2}d_X}$. By (\ref{eq-position3}), we have 
$$
\forall j> k ,\quad \Vert y_j\Vert_X\leqslant 1
$$
\vskip 0.3cm  
\textbf{-Fourth step}: Define $T : l_1^k\longrightarrow X$ by 
$T(e_j)= x_j$ if $j\leqslant k$ and $T(e_j)=y_j$ if $j> k$. 
Let $a=(a_j)_{j\leqslant n} \in \mathbb{R}^n$ and write 
$$
Ta=
\displaystyle \sum_{j=1}^k a_jx_j+
\sum_{j=k+1}^n a_j y_j.
$$ 
\vskip 0.3cm 
Then using the triangle inequality and (\ref{eq-position3}), one can write 
$$
\Vert a\Vert_1 = \displaystyle
\sum_{j\leqslant k} \vert a_j\vert + \sum_{j> k } \vert a_j\vert
               \geqslant \left\Vert \sum_{j\leqslant k } a_j x_j+ \sum_{j>k} a_j y_j\right\Vert_X
               \geqslant \left\Vert \sum_{j\leqslant k } a_j x_j + \sum_{j>k} a_j y_j\right\Vert_{\mathcal{E}_1}.\\ 
$$            

We also have
\begin{align*}
 \left\Vert Ta\right\Vert_{\mathcal{E}_1} &\geqslant \left[ \left\Vert \sum_{j\leqslant k } a_j x_j\right\Vert_{\mathcal{E}_1}^2 + \left\Vert \sum_{j>k} a_j y_j\right\Vert_{\mathcal{E}_1}^2\right]^{\frac{1}{2}} \quad \text{ by orthogonality}\\
               &\geqslant \left[ \frac{1}{2}\left\Vert \sum_{j\leqslant k } a_j x_j\right\Vert_{2}^2 + \left\Vert \sum_{j>k} a_j y_j\right\Vert_{\mathcal{E}_1}^2\right]^{\frac{1}{2}} \quad \text{ by (\ref{eq-position2})}\\
               &\geqslant \left[  \frac{1}{2}\varepsilon^2\sum_{j\leqslant k } a_j^2  +  \sum_{j>k} a_j^2 \Vert y_j\Vert_{\mathcal{E}_1}^2\right]^{\frac{1}{2}} \quad \text{ by } (\ref{eq-DR})\\
               &\geqslant \left[  \frac{\varepsilon^2}{2n}\left(\sum_{j\leqslant k } \vert a_j\vert\right)^2  +  \frac{1}{2d_X^2(n-k)}\left(\sum_{j>k} \vert a_j\vert\right)^2 \right]^{\frac{1}{2}} \text{ by Cauchy-Shwarz}\\
               &\geqslant \left[  \frac{\varepsilon^2}{2n}\left(\sum_{j\leqslant k } \vert a_j\vert\right)^2  +  \frac{1}{4\varepsilon nd_X^2}\left(\sum_{j>k} \vert a_j\vert\right)^2 \right]^{\frac{1}{2}}\\
               &\geqslant \frac{1}{2} \left[ \frac{\varepsilon}{\sqrt{n}} \sum_{j\leqslant k } \vert a_j\vert + \frac{1}{d_X\sqrt{2\varepsilon n}} \sum_{j>k} \vert a_j\vert\right]\\
               &\geqslant \frac{1}{2^{\frac{4}{3}}\sqrt{n} d_X^{\frac{2}{3}}} \sum_{j=1}^{n} \vert a_j\vert \quad \text{ taking } \varepsilon = (\sqrt{2}d_X)^{-\frac{2}{3}}.
\end{align*}
As a conclusion,
$$
\frac{1}{2^{\frac{4}{3}}\sqrt{n} d_X^{\frac{2}{3}}} \Vert a\Vert_1 \leqslant \Vert
Ta \Vert_X \leqslant \Vert a\Vert_1$$ and therefore $d(X,l_1^n)
\leqslant 2^{\frac{4}{3}}\sqrt{n} d_X^{\frac{2}{3}}$ for all $X\in \mathbb{BM}_n$.
\end{proof}

Using the same procedure and working only with one ellipsoid $\mathcal{F}$, the ellipsoid 
of minimal volume containing $B_X$, and noting that by John's theorem \cite{MR0030135} 
$\frac{1}{\sqrt{n}} \mathcal{F} \subset B_X\subset \mathcal{F}$, we get the following
\begin{theo}
$R_1^n =R_\infty^n \leqslant (2n)^{\frac{5}{6}}$.
\end{theo}

\begin{Rq}
\rm Here we are interested in high dimensional results; this is why
the constant is not that important. If we want an estimate
for ``small'' dimensions, then the value of the constant becomes
important. In \cite{MR1353450}, Giannopoulos proved that
$R_{\infty}^n \leqslant cn^{\frac{5}{6}}$ with $c=
\frac{2^{\frac{7}{6}}}{(\sqrt{2}-1)^{\frac{1}{3}}}\sim 3,0116$, and
thus his result becomes nontrivial when the dimension is larger than
$747$. On the other hand, our result becomes nontrivial whenever the
dimension is bigger than $32$. Moreover, we can obtain a better result for small dimensions 
by choosing
$\varepsilon$ in the last inequality in a different way: in fact we
have chosen $\varepsilon =(2n)^{-\frac{1}{3}}$ (replacing $d_X$ with $\sqrt{n}$) 
in the asymptotic regime,
otherwise one just need to optimize on $\varepsilon$ so that it
satisfies $\frac{\varepsilon}{\sqrt{(1-\varepsilon)^2n}} =
\frac{1}{n\sqrt{1-(1-\varepsilon)^2}}$; then our result becomes
nontrivial when the dimension is larger than $16$. In
\cite{MR2794363}, Taschuk has also obtained an estimate for the
Banach-Mazur distance to the cube of ``small''-dimensional spaces. 
Precisely, he proved the following
$$
R_\infty^n\leqslant \sqrt{n^2 -2n+2+\frac{2}{\sqrt{n+2}-1}}
$$
One can check that our result improves on that whenever the
dimension is larger than $22$.

\end{Rq}

\section{Projection on coordinate subspaces}

Given an $n\times m$ matrix $U$ and an integer $k\leqslant m$, 
our aim is to find a coordinate projection of $U$ of rank $k$ 
which gives the best minimal operator norm among all coordinate 
projections. First results were obtained by Lunin \cite{MR1001700}, 
and a complete answer to this question was given by Kashin-Tzafriri \cite{kashin-tzafriri} 
who proved the following:

\begin{theoE}[Kashin-Tzafriri]\label{kashin-tzafriri}
Let $U$ be an $n\times m$ matrix.
Fix $\lambda$ with $1/m \leqslant  \lambda  \leqslant \frac{1}{4}$. Then,
there exists a subset $\nu $ of $\{1, \ldots, m\}$ of cardinality
$|\nu| \geqslant  \lambda m $ such that
$$\| U_{\nu} \|  \leqslant  c \left( \sqrt{\lambda}\Vert U\Vert_2+\frac{\|U\|_{{\rm HS}}}{\sqrt{m}} \right),$$
where $U_\nu = UP_\nu$ and $P_\nu$ denotes the coordinate projection onto $\mathbb{R}^\nu$.
\end{theoE}

The conclusion of the Theorem states that for a fixed $\lambda<\frac{1}{4}$ we have 
\begin{equation}
\min_{\underset{\vert \sigma\vert =\lambda m}{\sigma\subset \{1,...,m\}}}
 \Vert U_{\sigma}\Vert \leqslant c \left( \sqrt{\lambda}\Vert U\Vert+
 \frac{\|U\|_{{\rm HS}}}{\sqrt{m}} \right),
\end{equation}
and this estimate is optimal 
in the sense that the dependence on the parameters in 
the right hand side cannot be improved.

Kashin-Tzafriri's proof (see \cite{MR1826503}) uses the selectors with 
some other probabilistic arguments and the Grothendieck's factorization Theorem. 
In \cite{MR2807539}, Tropp gave a randomized algorithm to realize 
Grothendieck's factorization theorem and therefore he was able 
to give a randomized algorithm to find the subset $\sigma$ promised 
in Theorem\,E. 

Our aim here is to give a deterministic algorithm to find the subset $\sigma$. 
Our method uses tools from the work of Batson-Spielman-Srivastava \cite{batson-spielman-srivastava} 
and allows us to improve Kashin-Tzafriri's result by getting better constants in the result and extending the size 
of the coordinate projection; indeed, in Theorem~E one can deal with a proportion of the columns 
less than $1/4$ while we will be able to work with any proportion smaller than $1$.

\begin{theo}\label{kashin-tzafriri-new}
Let $U$ be an $n\times m$ matrix and let $1/m \leqslant\lambda \leqslant \eta<1$. Then, there
exists $\sigma \subset \{1,\ldots , m\}$ with $\vert \sigma\vert = k
\geqslant \lambda m$ such that
$$\Vert U_\sigma\Vert \leqslant \frac{1}{\sqrt{1-\lambda}}
\left( \sqrt{\lambda +\eta} \Vert U\Vert + \sqrt{ 1+ \frac{\lambda}{\eta}}\frac{\Vert U\Vert_{{\rm
HS}}}{\sqrt{m}}\right),$$ 
In particular,
$$\Vert U_\sigma\Vert \leqslant \frac{\sqrt{2}}{\sqrt{1-\lambda}}
\left( \sqrt{\lambda} \Vert U\Vert +\frac{\Vert U\Vert_{{\rm
HS}}}{\sqrt{m}}\right),$$
where $U_\sigma$ denotes the selection of
the columns of $U$ with indices in $\sigma$.
\end{theo}

\begin{proof}
We denote by $(e_j)_{j\leqslant m}$ the canonical basis of
$\mathbb{R}^m$. Since $$U_\sigma\cdot U_\sigma^t = \displaystyle
\sum_{j\leqslant \sigma} \left( Ue_j\right)\cdot \left( Ue_j\right)^t,$$ our problem reduces to
the question of estimating the largest eigenvalue of this sum of
rank one matrices. We will follow the same procedure as in the proof
of the restricted invertibility theorem: at each step, we would like
to add a column of the original matrix and then study the evolution
of the largest eigenvalue. However, it will be convenient for us to
add suitable multiples of the columns of $U$ in order to construct
the $l$-th matrix; for each $l$ we will choose a subset $\sigma_k$
of cardinality $\vert \sigma_l\vert =l$ and consider the matrix 
$$
A_l=  \sum_{j\in\sigma_l} s_j \left(Ue_j\right)\cdot \left( Ue_j\right)^t,
$$ where
$(s_j)_{j\in\sigma}$ will be positive numbers which will be suitably
chosen. At the step $l$, the barrier will be denoted by $u_l$,
namely the eigenvalues of $A_l$ will be all smaller than $u_l$. The
corresponding potential is $\psi(A_l,u_l):= {\rm Tr}\left(U^t(u_lI
-A_l)^{-1}U\right)$. We set $A_0=0$, while $u_0$ will be determined
later.

As we did before, at each step the value of the potential
$\psi(A_l,u_l)$ will decrease so that we can continue the iteration,
while the value of the barrier will increase by a constant $\delta$,
i.e. $u_{l+1} = u_l +\delta $. We will use a lemma which appears as
Lemma\,3.4 in \cite{srivastava}. We state it here in the notation
introduced above.

\begin{lem}\label{lem-kashin}
Let $A$ be an $n\times n$ symmetric positive semidefinite matrix. 
Assume that $\lambda_{\max}(A) \leqslant u_l$. Let $v$ be a vector
in $\mathbb{R}^n$ satisfying
$$F_l(v) :=\frac{v^t(u_{ l+1}I-A)^{-2} v}{ \psi(A,u_{l})- \psi(A,u_{l+1})}
\Vert U\Vert^2 + v^t(u_{l+1}I-A)^{-1}v \leqslant \frac{1}{s}.$$
Then, if we define $A'= A + svv^t$ we have
$$\lambda_{\max}(A')\leqslant u_{l+1} \quad and\quad  \psi(A',u_{l+1})\leqslant \psi(A,u_l).$$
\end{lem}

\begin{proof}
Using Sherman-Morrison formula we have:
\begin{align*}
\psi(A',u_{l+1})&={\rm Tr}\left(U^t\left(u_{l+1}I-A-svv^t\right)U\right)\\
&={\rm Tr}\left(U^t \left(u_{l+1}I-A\right)U\right)+ 
\frac{sv^t(u_{l+1}I-A)^{-1}UU^t(u_{l+1}I-A)^{-1}v}{1-sv^t(u_{l+1}I-A)^{-1}v}\\
&\leqslant\psi(A,u_l)-\left(\psi(A,u_l)- \psi(A,u_{l+1})\right)+ \frac{v^t(u_{l+1}I-A)^{-2}v}{\frac{1}{s}-v^t(u_{l+1}I-A)^{-1}v}\Vert U\Vert^2\\
\end{align*}
Since $v^t(u_{l+1}I-A)^{-1}v<  F_l(v)$ and $F_l(v) \leqslant \frac{1}{s}$ 
we deduce that the quantity above is finite. This implies that $\lambda_{\rm max}(A')< u_{l+1}$, since 
otherwise one would find $s'<s$ such that $\lambda_{\rm max}(A+s'vv^t)=u_{l+1}$ and therefore 
$\psi(A+s'vv^t,u_{l+1})$ would blow up which contradicts the fact that it is finite.\\
On the other hand, rearranging the inequality above using the fact that  $F_l(v) \leqslant \frac{1}{s}$ 
we get $\psi(A',u_{l+1})\leqslant \psi(A,u_l).$
\end{proof}

We write $\alpha$ for the initial potential, i.e. $\alpha =
\frac{\Vert U\Vert_{{\rm HS}}^2}{u_0}$ . Suppose that
$A_l=\sum_{j\in\sigma_l} s_j\left(Ue_j\right)\cdot \left(Ue_j\right)^t$ is
constructed so that $\psi(A_l,u_l) \leqslant \psi(A_{l-1},u_{l-1})
\leqslant \alpha$ and $\lambda_{\max}(A_l) \leqslant u_l$. We will
now use Lemma\,\ref{lem-kashin} in order to construct $A_{l+1}$. To
this end, we must find a vector $Ue_j$ not chosen before and a
scalar $s_{l+1}$ so that $F_l(Ue_j) \leqslant
\frac{1}{s_{l+1}}$, and then use Lemma~\ref{lem-kashin}. 
Since $(u_lI - A_l)^{-1}$ and $(u_{l+1}I -A_l)^{-1}$ are
positive semidefinite, one can easily check that
$$(u_lI - A_l)^{-1}-(u_{l+1}I- A_l)^{-1}\succeq \delta (u_{l+1}I -A_l)^{-2}.$$
Therefore,
$${\rm Tr}\left(U^t(u_{l+1}I-A_l)^{-2}U\right)\leqslant \frac{1}{\delta}
\left( \psi(A_l,u_{l}) -\psi(A_l,u_{l+1})\right).$$ It follows that
\begin{align*}
\displaystyle \sum_{j\not\in\sigma_l} F_l(Ue_j) &\leqslant \sum_{j\leqslant m} F_l(Ue_j)
= \frac{Tr\left( U^t (u_{l+1}I-A_l)^{-2}U\right)}{\psi(A_l,u_l)- \psi(A_l,u_{l+1})} \Vert U\Vert^2 + \psi(A_l,u_{l+1})\\
&\leqslant \frac{\Vert U\Vert^2}{\delta}+ \alpha,
\end{align*}
and therefore one can find $i\not\in\sigma_l$ such that
\begin{equation}\label{last-eq-kashin}
F_l(Ue_i)\leqslant \frac{1}{\vert \sigma_l^c\vert} \left(
\frac{\Vert U\Vert^2}{\delta}+ \alpha\right)\leqslant
\frac{1}{\vert \sigma_k^c\vert} \left(  \frac{\Vert
U\Vert^2}{\delta}+ \alpha\right),
\end{equation}
 where $k$ is the maximum
number of steps (which is in our case $\lambda m$). \\We are going to
choose all $s_j$ equal to $s :=\frac{(1-\lambda) m }{\alpha +
\frac{\Vert U\Vert^2}{\delta}}$. With this choice of $s$ and by (\ref{last-eq-kashin}), the condition of Lemma~\ref{lem-kashin} 
is satisfied and therefore we are able to construct $A_{l+1}$. After $k=
\lambda m$ steps, we get $\sigma = \sigma_k$ such that
\begin{align*}
\displaystyle \lambda_{\mathrm{max}}\left(\sum_{j\in\sigma_k} \left(Ue_j\right)\cdot\left(
Ue_j\right)^t \right)&\leqslant \frac{1}{s} u_k=\frac{1}{s} (u_0 + k \delta )
= \frac{ \alpha + \frac{\Vert U\Vert^2}{\delta}}{(1-\lambda) m} \left(u_0 + k\delta\right)\\
&= \frac{1}{1-\lambda}\left[\frac{\Vert U\Vert_{{\rm HS}}^2}{m} +
\lambda \Vert U\Vert^2
+\lambda\Vert U\Vert_{{\rm HS}}^2 \frac{\delta}{u_0} + \frac{\Vert U\Vert^2}{m} \frac{u_0}{\delta}\right]
\end{align*}
The result follows by taking $u_0= \eta m \delta$. The second part of the theorem follows by taking $\lambda = \eta$.
\end{proof}

\bigskip

\noindent \textbf{Aknowledgement}. I am grateful to my PhD
advisor Olivier Guédon for many helpful discussions. I would also like to thank the doctoral
school of Paris-Est for giving me the opportunity to visit the University
of Athens and Apostolos Giannopoulos for his hospitality and his precious help. I would like 
to thank the anonymous referee for his valuable remarks.
\nocite{*}
\bibliographystyle{abbrv}
\bibliography{bibliography-invertibility}

\end{document}